\documentclass[11pt, english, draft]{article}
\usepackage{amssymb,amsmath}
\usepackage{asymptote}
\usepackage{xcolor} % package crucial to implement colors  
\usepackage{pgfplots}
\usepackage{pgfplotstable}
\pgfplotsset{compat=1.7}
\usepackage{pgf,tikz}
\usepackage{mathrsfs}
\usetikzlibrary{arrows}
\usepackage{tikz}
\usepackage{rotating}
\usepackage{tabularx,ragged2e,booktabs,caption}
\usepackage{floatrow}
\usepackage{array}
\usepackage{booktabs}% for top-,mid- & bottomrule
\usepackage{caption}

\usetikzlibrary{calc,positioning}
\pgfmathsetseed{\pdfuniformdeviate 10000000}

\usepackage{imakeidx}
\indexsetup{othercode=\small}
\usepackage[font=footnotesize]{idxlayout}

\newenvironment{proof}{\noindent {\em {Proof}}.}{$\square$
\medskip}
\newtheorem{Theorem}{Theorem}[section]

\newtheorem{Corollary}{Corollary}[section]

\newtheorem{Lemma}{Lemma}[section]

\usepackage{amsmath}

\newcommand{\rvline}{\hspace*{-\arraycolsep}\vline\hspace*{-\arraycolsep}}

\begin{document}

\setlength{\unitlength}{10mm}
\title{\bf On the tree-number of the power graph associated with a finite groups}

 %\author{\ {\Sc A. Azad and S. Rahbariyan} \\[0.3cm]
%{Department of Mathematics, Faculty of Sciences, Arak University}\\[0.1cm]
%\email{}\\[0.2cm]}
\author{\ {\sc  S. Rahbariyan }
\\[0.3cm]
{\em Department of Mathematics, Faculty of Sciences,}\\[0.1cm]
 {\em Arak University}\\[0.1cm]
 {\em P. O. Box $ 38156-8-8349$, Arak, Iran}\\[0.3cm]
{\em E-mail}: {\tt    s-rahbariyan@araku.ac.ir  }\\[0.2cm]}
\maketitle

\date{}
\maketitle
\begin{abstract}

Given a group $G$, we define the power graph $\mathcal{P}(G)$ as follows: the vertices are the elements of $G$
and two vertices $x$ and $y$ are joined by an edge if $\langle x \rangle \subseteq \langle y \rangle$ or 
$\langle y \rangle \subseteq \langle x \rangle$. Obviously the power
graph of any group is always connected, because the identity element of the group is adjacent
to all other vertices. We consider $\kappa(G)$, the number of
spanning trees of the power graph associated with a finite group $G$.
In this paper, for a finite group $G$, first we represent some properties of $\mathcal{P}(G)$,  
 then we are going to find some divisors of $\kappa(G)$, and finally we prove that 
 the simple group $A_6\cong L_2(9)$  is uniquely determined by
tree-number of its power graph among all finite simple groups.
			\end{abstract}

	\section{Introduction}
Throughout this paper, all groups are finite and  all the graphs under
consideration are finite, simple (with no loops or multiple edges) and undirected.
 In this paper, we consider a well-known graph association with  a finite group named as power graph. For a group $G$,  the power graph  $\mathcal{P}(G)$, is the graph that its vertices are all elements of the group $G$ and two different vertices 
 $x$ and $y$ are joined by an edge if $\langle x \rangle \subseteq \langle y \rangle$ or 
$\langle y \rangle \subseteq \langle x \rangle$. We
denote by $\mathcal{P}^*(G)$  the graph obtained by deleting the vertex $1$ from $\mathcal{P}(G)$.
The term power graph was introduced in \cite{KQ}, and after that power graphs
have been investigated by many authors, see for instance \cite{AK, C, MRS}. The
investigation of power graphs associated with algebraic structures is important, because these graphs have valuable applications (see the survey article
\cite{KR}) and are related to automata theory (see the book \cite{K}).

A spanning
tree of a connected graph is a subgraph that contains all the vertices and
is a tree. Counting the number of spanning trees in a connected graph is
a problem of long-standing interest in various field of science. For a graph
$\Gamma$, the number of spanning trees of $\Gamma$; denoted by $\kappa(\Gamma)$; is known as the
complexity of $\Gamma$.
 By the  definition
of the power graph of any group, the identity element of the group
is adjacent to all other vertices, so the graph is always connected. We denote by $\kappa(G)$, the number of spanning trees of the power
graph $\mathcal{P}(G)$ of a group $G$. 
A well-known result due to Cayley \cite{Ca} says that the complexity of
the complete graph on $n$ vertices is $n^{n-2}$. In \cite{CH} it was shown that a finite group has a complete power graph if and only if it is a cyclic $p$-group, where $p$
is a prime number. Thus, as an immediate consequence of Cayley's result, we
derive $\kappa(\mathbb{Z}_p^m)=(p^m)^{p^m-2}$. In some investigations, the formula to compute the complexity
$\kappa(G)$, for instance,  where $G$ is the cyclic group $\mathbb{Z}_n$, dihedral group $D_{2n}$, the generalized quaternion group $Q_{4n}$ (see  \cite{MR}), the simple groups $L_2(q)$, the extra-special $p$-groups of order $p^3$ and
the Frobenius groups (see \cite {KI}) have been obtained. In this paper, we take a step forward and find some  general results due to complexity
$\kappa(G)$, for a finite group $G$, and finally, as an application of these results, we represent the following investigation.

For two isomorphic groups $G$ and $H$, clearly, $\kappa(G)=\kappa(H)$. However, generally the converse is not hold. 
 For
instance, for all finite elementary abelian $2$-groups $G$,  we have $\kappa(G)=1$.

A group $G$ from a class $\mathcal{C}$ is said to be recognizable in $\mathcal{C}$ by $\kappa(G)$ (shortly, $\kappa$-recognizable in $\mathcal{C}$)  if every group $H\in \mathcal{C}$ with $\kappa(H)=\kappa(G)$ is isomorphic to $G$. In other words, $G$ is $\kappa$-recognizable in
$\mathcal{C}$ if $h_{\mathcal{C}}(G) = 1$, where $h_{\mathcal{C}}(G)$ is the (possibly infinite) number of pairwise non-isomorphic groups   $H\in \mathcal{C}$ with $\kappa(H)=\kappa(G)$.
 We denote by $\mathcal{F}$ and $\mathcal{S}$ the classes of all finite groups and all finite
simple groups, respectively. In \cite{MR}, the first example of $\kappa$-recognizable group in class $\mathcal{S}$ was found.

\begin{Theorem}\cite{MR} \label{AA5}The alternating group $A_5$  is $\kappa$-recognizable in the class of all finite simple groups,
that is, $h_{\mathcal{S}}(A_5) = 1$.
\end{Theorem}

 After that, in \cite{KI}, the $\kappa$-recognizable group in class $\mathcal {S}$ for simple group $L_2(7)$, also, has been proven.
However, by the results as to complexity $\kappa(G)$, which are found in this paper, we are going to offer a different and short  proof for the following theorem.  

\begin{Theorem} \label{A6} For the simple group $A_6\cong L_2(9)$, we have $h_{\mathcal {S}} (A_6\cong L_2(9)) = 1$,
in the class $\mathcal{S}$ of all
finite simple groups.
\end{Theorem}

%%%%%%%%%%%%%%%%%%%%%%%%%%%%%%%%%%%%%%%%%%%%%%%%

\section{Terminology and Previous Results}

The notation and definitions used in this paper are standard and taken mainly from \cite{Bi, atlas, Ro, WE}.
We will cite only a few. Let $\Gamma= (V,E)$  be a simple graph. We denote by $\mathbf{A} = \mathbf{A}(\Gamma)$  the adjacency
matrix of $\Gamma$. The Laplacian matrix $\mathbf{Q}$ of a graph $\Gamma$  is $\mathbf{\Delta}- \mathbf{A}$, where  $\mathbf{\Delta}$  is
the diagonal matrix whose $i$-th diagonal entry is the degree $v_i$ in $\Gamma$ and $\mathbf{A}$ is the adjacency matrix
of $\Gamma$. The $\mathbf{J}_{m\times n}$ and $\mathbf{O}_{m\times n}$ denote the matrixes with $m$ rows and $n$ columns, where each of whose entries is +1, and zero, respectively. Moreover, the  identity matrix is denoted by    $\mathbf{I}$.
A matrix $\mathcal{A}$ of size $n$, which is the $n\times n$ square matrix, is denoted by $\mathcal{A}_{n\times n}$. 

When $U\subseteq V$, the induced subgraph $\Gamma[U]$ is the subgraph
of $\Gamma$ whose vertex set is $U$ and whose edges are precisely the edges of $\Gamma$ which have both ends in $U$.
Two graphs are disjoint if they have no vertex in common, and edge-disjoint if they have no edge
in common. If $\Gamma_1$  and $\Gamma_2$  are disjoint, we refer to their union as a disjoint union, and generally
denote it by $\Gamma_1\oplus\Gamma_2$. By starting with a disjoint union of two graphs $\Gamma_1$  and $\Gamma_2$ and adding edges joining every vertex of $\Gamma_1$ to every vertex of $\Gamma_2$, one obtains the join of $\Gamma_1$ and $\Gamma_2$,  denoted $\Gamma_1\vee\Gamma_2$. A clique in a graph is a set of pairwise adjacent vertices. 

We denote by $\pi(n)$ the set of the prime divisors of a positive integer $n$. Given
a group $G$, we will write $\pi(G)$  instead of $\pi(|G|)$, and  denote by $\pi_e(G)$ the
set of orders of all elements in a group $G$ and call this set the spectrum of $G$. The spectrum $\pi_e(G)$
of $G$ is closed under divisibility and determined uniquely from the set $\mu(G)$ of those elements in
$\pi_e(G)$ that are maximal under the divisibility relation. In the case when $\mu(G)$ is a one-element set
$\{n\}$, we write $\mu(G) = n$. Finally, two
notation  $\phi(n)$ and  $c_m(G)$ are denoted in particular for
the  Euler's totient function, for positive integer $n$, and the
number of distinct  cyclic subgroups of order $m$ of $G$.
Occasionally, when the group we are considering is clear from the
context, we will simply write $c_m$ instead of $c_m(G)$.  All further
unexplained notation is standard and refers to \cite{atlas}.

At the following,  we give several auxiliary results to be used later. First, we point some important 
lemmas due to the power graph. 

\begin{Lemma}\cite{CH}\label{CH1} Let $G$ be a finite group. Then $\mathcal{P}(G)$ is complete
if and only if $G$ is a cyclic group of order $1$ or $p^m$ for some prime number $p$ and
for some natural number $m$.
\end{Lemma}

\begin{Lemma}\cite{MRS}\label{MRS1} Let $G$ be a finite $p$-group, where $p$ is a prime. Then $\mathcal{P}^*(G)$ is connected
if and only if $G$ has a unique minimal subgroup.
\end{Lemma}

\begin{Corollary}\cite{MRS}\label{MRS2} Let $G$ be a finite $p$-group, where $p$ is a prime. Then $\mathcal{P}^*(G)$ is
connected if and only if $G$ is either cyclic or generalized quaternion.
\end{Corollary}

\begin{Lemma}\cite {MRS} \label{MRS2} Let $G$ be a finite group. If $H$ is a subgroup of $G$, 
then $\mathcal{P}(H)$  is a subgraph of $\mathcal{P}(G)$. In particular, if $x$
is a $p$-element of $G$, where $p$ is a prime, then $\langle x \rangle$  is a clique in $\mathcal{P}(G)$.
\end{Lemma}

In the sequel, we collected some results related to  the number of spanning trees of a simple graph $\Gamma$. 
The following one is well known, see for example [\cite{WE}, Proposition 2.2.8].

\begin{Theorem}\cite{WE}\label{WE1}(Deletion-Contraction Theorem) The number of spanning trees of a graph $\Gamma$
satisfies the deletion-contraction recurrence $\kappa(\Gamma)=\kappa(\Gamma-e)+\kappa(\Gamma\cdot e)$,
 where $e\in E(\Gamma)$. In particular,
if $e\in E(\Gamma)$ is a cut-edge, then $\kappa(\Gamma)=\kappa(\Gamma\cdot e)$.
\end{Theorem}

\begin{Theorem}\cite{MR}\label{MR1}  Let $\Gamma$ be a  connected graph and let $v$  be a cut vertex of $\Gamma$ with
$$\Gamma-v=\Gamma_1\oplus \Gamma_2\oplus\cdots \oplus \Gamma_c,$$
where $\Gamma_i$, $i=1, 2, \ldots, c$, is the ith connected component of $\Gamma-v $ and $c=c(\Gamma-v)$. Set $\tilde{\Gamma_i}=\Gamma_i+v$. 
Then, there holds
$$\kappa(\Gamma)= \kappa(\tilde{\Gamma_1})\times\kappa(\tilde{\Gamma_2})\times\cdots \kappa(\tilde{\Gamma_c}).$$
\end{Theorem}

\begin{Lemma}\cite{MR}\label{MR} If $H_1$, $H_2$, . . . , $H_t$ are nontrivial subgroups of
a group $G$ such that $H_i \cap H_j = \{1\}$, for each $1\leqslant i <j\leqslant t$, then we have
$\kappa(G) >\kappa(H_1)\cdot \kappa(H_2)\cdots \kappa(H_t)$.
\end{Lemma}

\begin{Corollary}\cite{MR}\label{MR4} Let $G$ be a finite group and let $p$ be the smallest prime such that $\kappa(G) <p^{(p-2)}$ 
Then $\pi(G)\subseteq \pi((p-1)!)$.
\end{Corollary}

\begin{Theorem}\cite{TE}\label{TE1} The number of spanning trees of a graph $\Gamma$ with $n$ vertices is
given by the formula
$$\kappa(\Gamma)=\det(\mathbf{J}+\mathbf{Q})/n^2,$$
where $\mathbf{J}$ denotes the matrix each of whose entries is $+1$.
\end{Theorem}

%{\bf Remark 1.} If $G$ is a finite group, then by the dififnition of $\mathcal{P}(G)$, since $1$ has the full degree, clearly
%$|G| \ | \ \det(\mathbf{J}+\mathbf{Q})$.

For instance, we consider the power graphs of quaternion group $$Q_8=\langle x, y \ | \ x^4=1, \  x^2=y^2, \ yx=x^{-1}y\rangle,$$ and apply Theorem \ref{TE1}, to find $\kappa(Q_8)$. The power graph  $\mathcal{P}(Q_8)$   is shown in Fig. 1. 

\vspace{1cm}
%{\fontsize{8}{10}\selectfont
\begin{center}
\begin{tikzpicture}
%% vertices
\draw[fill=black] (5,3) circle (3pt);
\draw[fill=black] (5,2) circle (3pt);
\draw[fill=black] (9,3) circle (3pt);
\draw[fill=black] (9,2) circle (3pt);
\draw[fill=black] (6.2,-1) circle (3pt);
\draw[fill=black] (7.8,-1) circle (3pt);
\draw[fill=black] (7,1) circle (3pt);
\draw[fill=black] (7,2) circle (3pt);

%% vertex labels
\node at (7,2.5) {$x^2$};
\node at (7,0.5) {1};
\node at (4.5,2) {$y$};
\node at (4.5, 3) {$yx^2$};
\node at (9.5, 3) {$yx^3$};
\node at (9.5, 2) {$yx$};
\node at (5.7, -1) {$x$};
\node at (8.3, -1) {$x^3$};

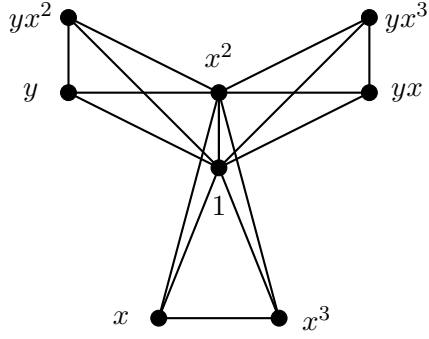
\node at (7, -3) {\bf{Fig. 1. The graph  $\mathcal{P}(Q_{8})$}};

%%% edges
\draw[thick] (7,1)--(5,3);
\draw[thick] (7,1)--(5,2);
\draw[thick] (7,1)--(9,3);
\draw[thick] (7,1)--(9,2);
\draw[thick] (7,1)--(6.2,-1);
\draw[thick] (7,1)--(7.8,-1);
\draw[thick] (7,1)--(7, 2);
\draw[thick] (7,2)--(5,3);
\draw[thick] (7,2)--(5,2);
\draw[thick] (7,2)--(9,3);
\draw[thick] (7,2)--(9,2);
\draw[thick] (7,2)--(6.2,-1);
\draw[thick] (7,2)--(7.8,-1);
\draw[thick] (7,2)--(7, 1);
\draw[thick]  (5,3)--(5, 2);
\draw[thick] (9,3)--(9, 2);
\draw[thick](6.2,-1)--(7.8,-1);
\end{tikzpicture}
\end{center}

For this graph,  by the definition,  the adjacency matrix $\mathbf{A}$ and the diagonal matrix $\mathbf{\Delta}$ have  the following
 structures:

$$
\mathbf{A}=\left(
\begin{array}{c | c | c|c|c} 
\begin{array}{c c}
0& 1
\end{array}
 &\begin{array}{c c}
0& 0
\end{array}&
\begin{array}{c c}
0& 0
\end{array}&
1& 1\\
\begin{array}{c c}
1& 0
\end{array}
 &\begin{array}{c c}
0& 0
\end{array}&
\begin{array}{c c}
0& 0
\end{array}&
1& 1\\
\hline
\begin{array}{c c}
0& 0
\end{array}
 &\begin{array}{c c}
0& 1
\end{array}&
\begin{array}{c c}
0& 0
\end{array}&
1& 1\\
\begin{array}{c c}
0& 0
\end{array}
 &\begin{array}{c c}
1& 0
\end{array}&
\begin{array}{c c}
0& 0
\end{array}&
1& 1\\
\hline
\begin{array}{c c}
0& 0
\end{array}
 &\begin{array}{c c}
0& 0
\end{array}&
\begin{array}{c c}
0& 1
\end{array}&
1& 1\\
\begin{array}{c c}
0& 0
\end{array}
 &\begin{array}{c c}
0& 0
\end{array}&
\begin{array}{c c}
1& 0
\end{array}&
1& 1\\
\hline
\begin{array}{c c}
1& 1
\end{array}
 &\begin{array}{c c}
1& 1
\end{array}&
\begin{array}{c c}
1& 1
\end{array}&
0& 1\\
\hline
\begin{array}{c c}
1& 1
\end{array}
 &\begin{array}{c c}
1& 1
\end{array}&
\begin{array}{c c}
1& 1
\end{array}&
1& 0\\
 \end{array} 
\right), \ \ \ \ \ \ \ \& \ \ \ \ \ \ \ 
\mathbf{\Delta}=\left(
\begin{array}{c | c | c|c|c} 
\begin{array}{c c}
3& 0
\end{array}
 &\begin{array}{c c}
0& 0
\end{array}&
\begin{array}{c c}
0& 0
\end{array}&
0& 0\\
\begin{array}{c c}
0& 3
\end{array}
 &\begin{array}{c c}
0& 0
\end{array}&
\begin{array}{c c}
0& 0
\end{array}&
0& 0\\
\hline
\begin{array}{c c}
0& 0
\end{array}
 &\begin{array}{c c}
3& 0
\end{array}&
\begin{array}{c c}
0& 0
\end{array}&
0& 0\\
\begin{array}{c c}
0& 0
\end{array}
 &\begin{array}{c c}
0& 3
\end{array}&
\begin{array}{c c}
0& 0
\end{array}&
0& 0\\
\hline
\begin{array}{c c}
0& 0
\end{array}
 &\begin{array}{c c}
0& 0
\end{array}&
\begin{array}{c c}
3& 0
\end{array}&
0& 0\\
\begin{array}{c c}
0& 0
\end{array}
 &\begin{array}{c c}
0& 0
\end{array}&
\begin{array}{c c}
0& 3
\end{array}&
0& 0\\
\hline
\begin{array}{c c}
0& 0
\end{array}
 &\begin{array}{c c}
0& 0
\end{array}&
\begin{array}{c c}
0& 0
\end{array}&
7& 0\\
\hline
\begin{array}{c c}
0& 0
\end{array}
 &\begin{array}{c c}
0& 0
\end{array}&
\begin{array}{c c}
0& 0
\end{array}&
0& 7\\
 \end{array} 
\right) 
.$$
\vspace{0.2cm}

Therefore 

$$
\mathbf{J}+\mathbf{Q}=\mathbf{J}+(\mathbf{\Delta}-\mathbf{A})=\left(
\begin{array}{c | c | c|c|c} 
\begin{array}{c c}
4& 0
\end{array}
 &\begin{array}{c c}
1& 1
\end{array}&
\begin{array}{c c}
1& 1
\end{array}&
0&0\\
\begin{array}{c c}
0& 4
\end{array}
 &\begin{array}{c c}
1& 1
\end{array}&
\begin{array}{c c}
1&1
\end{array}&
0& 0\\
\hline
\begin{array}{c c}
1&1
\end{array}
 &\begin{array}{c c}
4& 0
\end{array}&
\begin{array}{c c}
1& 1
\end{array}&
0& 0\\
\begin{array}{c c}
1& 1
\end{array}
 &\begin{array}{c c}
0& 4
\end{array}&
\begin{array}{c c}
1& 1
\end{array}&
0& 0\\
\hline
\begin{array}{c c}
1&1
\end{array}
 &\begin{array}{c c}
1&1
\end{array}&
\begin{array}{c c}
4& 0
\end{array}&
0& 0\\
\begin{array}{c c}
1& 1
\end{array}
 &\begin{array}{c c}
1& 1
\end{array}&
\begin{array}{c c}
0& 4
\end{array}&
0& 0\\
\hline
\begin{array}{c c}
0& 0
\end{array}
 &\begin{array}{c c}
0& 0
\end{array}&
\begin{array}{c c}
0& 0
\end{array}&
8& 0\\
\hline
\begin{array}{c c}
0& 0
\end{array}
 &\begin{array}{c c}
0& 0
\end{array}&
\begin{array}{c c}
0& 0
\end{array}&
0& 8\\
 \end{array} 
\right) 
,$$

\vspace{0.3cm}

So by Theorem \ref{TE1} and easy calculation, $\kappa(Q_8)=\frac{\det(\mathbf{J}+\mathbf{Q})}{8^2}= 2^{11}$, as we expected by the following Theorem. 

\begin{Theorem}\cite{MR}\label{MR5}
If $n$ is a power of $2$, then the tree-number of the power graph $\mathcal{P}(Q_{4n})$  is given by
the formula $\kappa(Q_{4n})=2^{5n-1}\cdot n^{2n-2}$.
\end{Theorem}

We conclude this section with two results which are used for our final main theorem (Theorem \ref{A6}). 

\begin{Theorem} \label{mnn}\cite{CM} Let $q = p^n$, with $p$ prime and $n\in \mathbb{N}$, let $G = L_2(q)$. Then
we have:
$$ \kappa(G) = p^{\frac{(q^2-1)(p-2)}{p-1}}\cdot \kappa(\mathbb{Z}_{\frac{q-1}{k}})^{q(q+1)/2}\cdot \kappa(\mathbb{Z}_{\frac{q+1}{k}})^{q(q-1)/2},$$
where $k ={\rm  gcd}(q-1, 2)$, except exactly in the cases $(p, n) = (2, 1), (3, 1)$. In
particular, we have
\begin{item}
\item $A_5\cong  L_2(5)\cong  L_2(4)$ and $\kappa(A_5) = 3^{10}\cdot 5^{18}$ (see \cite{MR})

\item $ L_3(2)\cong L_2(7) $ and $\kappa(L_3(2)) = 2^{84}\cdot3^{28}
\cdot7^{40}$

\item $A_6\cong L_2(9)$ and $\kappa(A_6) = 2^{180}\cdot3^{40}\cdot5^{108}$.
\end{item}
\end{Theorem}

\begin{Lemma}\cite{MR}\label{MR1} Let $G$ be a finite nonabelian simple group and let
$p$ be a prime dividing the order of $G$. Then $G$ has at least $p^2-1$ elements of
order $p$, or equivalently, there is at least $p+1$ cyclic subgroups of order $p$ in $G$.
\end{Lemma}

 %%%%%%%%%%%%%%%%%%%%%%%%%%%%%%%%%%%%%%%%%%%%%%

	\section{Main parts of manuscrips}

%%%%%%%%%%%%%%%%%%%%%%%%%%%%%%%%%%%%%%%%%%%%%%%%%%%%%%%

As we mentioned before,  by the  definition
of the power graph of any group, the identity element of the group
is adjacent to all other vertices, so the graph is always connected. In this section, first, we look deeper to the power graph associated with a finite group and prove some necessarily lemmas, and then  we are going to find some useful divisors of the $\kappa(G)$, for a finite group $G$.

\begin{Lemma}\label{B1} Let $G$ be a $p$-group, where $p$ is a prime. Then $\mathcal{P}^*(G)$ has exactly $c_p$ connected
component, where $c_p$ is the number of distinct cyclic subgroups of order $p$ of $G$.
\end{Lemma}
\begin{proof} If $\mathcal{P}^*(G)$ is connected, then by Lemma \ref{MRS1}, $G$ has a unique minimal subgroup ($c_p=1$), and so there is nothing to be proved. Therefore we assume that $\mathcal{P}^*(G)$ be disconnected. Obviously, by the definition of $\mathcal{P}(G)$, 
in every connected component of  $\mathcal{P}^*(G)$, there must be at least one cyclic subgroup of order $p$. Assume that there are two distinct subgroups $\langle x \rangle$ and $\langle y \rangle$ in some connected components of  $\mathcal{P}^*(G)$. Let the following path, be a shortest path from $x$ to $y$,
$$x=x_0\sim x_1\sim x_2\sim \dots \sim x_n=y.$$ 
Certainly, $n \geq  2$ and $x \nsim x_i$, for each $i = 2, 3, . . ., n$.
Since $x\sim  x_1$, by the definition, we have $ x\in \langle x_1 \rangle$ or $ x_1\in \langle x \rangle$.
 We only consider the
first case, and the second one goes similarly. Since $x_1 \sim x_2$, it follows that $ x_1\in \langle x_2 \rangle$
or $ x_2\in \langle x_1 \rangle$. If $ x_1\in \langle x_2 \rangle$, then $ x\in \langle x_2 \rangle$
 which is a contradiction. Therefore, $ x_2\in \langle x_1 \rangle$
 But then $x, x_2\in \langle x_1\rangle$ and since $\langle x_1 \rangle$ is a $p$-group, Lemma \ref{MRS2} shows
that $\langle x_1 \rangle\setminus \{1\}$ is a clique in $\mathcal{P}^*(G)$. Hence $x \sim x_2$, a contradiction again. This
completes the proof. 
\end{proof}

An immediate consequence of  Theorem \ref{MR1} and Lemma \ref{B1} is the
following:

\begin{Corollary}\label{B3} Let $G$ be a $p$-group, for some prime numbers $p$, and $H_i$ be the ith connected component of 
$\mathcal{P}^*(G)$. 
Then  
$$\kappa(G)= \kappa(\tilde{H_1})\times\kappa(\tilde{H_2})\times\cdots\times \kappa(\tilde{H_{c_p}}),$$
where  $\tilde{H_i}=H_i+1$, for $i = 1, 2, \dots c_p$.
\end{Corollary}

Clearly if $H$ is a subgroup of $G$, then $\mathcal{P}^*(H)$ is a subgraph of $\mathcal{P}^*(G)$. However, the converse
is not true, even if we consider the vertices of a connected components of $\mathcal{P}^*(G)$ in union with $\{1\}$. For instance,  by the subgraph  $\mathcal{P}^*(\mathbb{Z}_2\times\mathbb{Z}_4)$ (see Fig. 2.), there is no subgroup $H$ of $G$, in which, $\mathcal{P}^*(H)$ be the connected component with $5$ vertices.

\vspace{1cm}
%{\fontsize{8}{10}\selectfont
\begin{center}
\begin{tikzpicture}
%% vertices
\draw[fill=black] (6,4) circle (3pt);
\draw[fill=black] (5,3) circle (3pt);
\draw[fill=black] (8,4) circle (3pt);
\draw[fill=black] (9,3) circle (3pt);
\draw[fill=black] (7,2) circle (3pt);
\draw[fill=black] (6,0) circle (3pt);
\draw[fill=black] (8,0) circle (3pt);

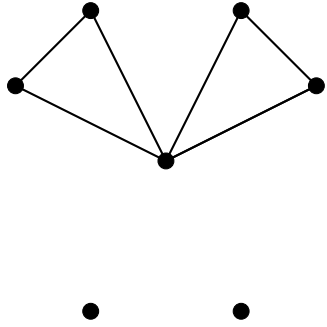
\node at (7, -1) {\bf{Fig. 2. The subgraph  $\mathcal{P}^*(\mathbb{Z}_2\times\mathbb{Z}_4)$}};

%%% edges
\draw[thick] (5,3)--(6,4);
\draw[thick] (9,3)--(8,4);
\draw[thick] (7,2)--(9,3);
\draw[thick] (7,2)--(8,4);
\draw[thick] (7,2)--(6,4);
\draw[thick] (7,2)--(5,3);
\draw[thick] (7,2)--(9,3);
\end{tikzpicture}
\end{center}

 At the following, we show that in a particular situation, the converse could be hold. 

\begin{Lemma}\label{B6} Let $G$ be a group and  $\Omega$ be a connected component of $\mathcal{P}^*(G)$
which is a clique. Then there is a cyclic $p$-subgroup $H$ of $G$ with $p\in \pi(G)$ that $\mathcal{P}^*(H)=\Omega $. 
\end{Lemma}
\begin{proof} By the definition of power graph, if $(o(g_1), o(g_2))=1$, for $g_1, g_2 \in G$, then $g_1 \nsim g_2$ in 
$\mathcal{P}(G)$. Therefore, since $\Omega$ is a clique,  for every element $\alpha$ in $V(\Omega)$, we must have $\pi(o(\alpha))=p$, for a prime number 
$p$. Let $\beta\in V(\Omega)$ be the vertex which its order has the largest power of the prime p. We claime that $\Omega=\mathcal{P}^*(\langle \beta \rangle)$. For every element $\alpha$ in $V(\Omega)$, 
 if $\alpha \in \langle \beta \rangle$, then there is nothing to be proved. Let $\beta \in \langle \alpha \rangle$, but then $o(\beta) \ | \ o(\alpha)$, which implies that $o(\alpha) = o(\beta)$, because $o(\beta)$ is the largest power of the prime p, and so $\langle \alpha \rangle= \langle \beta \rangle$. This proves our claim. 
\end{proof}

\begin{Corollary}\label{B9} Let $G$ be a $p$-group, for prime number $p$. If all connected components of  $\mathcal{P}^*(G)$
are clique, then $\pi(\kappa(G))=\{p\}$.
\end{Corollary}
\begin{proof} The proof is straightforward, by Lemma \ref{B6} and Theorem \ref{MR1}. 
\end{proof}

Now, we are ready to represent our results due to some useful divisors of the $\kappa(G)$, for a finite group $G$. 

\begin{Lemma}\label{B2} Let $G$ be a group and $p\in \mu(G)$, for some prime number $p$ in $\pi(G)$, then $p^{p-2} \ | \ \kappa(G)$.
\end{Lemma}
\begin{proof} Let $g \in G$ be an element of order $p$, for a prime number $p\in \mu(G)$. By Lemma \ref{CH1}, $\mathcal{P}(\langle g \rangle)$ is a complete graph. Hence, we only need to prove $\langle g \rangle$ is a connected component in $\mathcal{P}^*(G)$.
Assume that $g$ and $h$ are adjacent, for some $h \in G$. If $h \in \langle g \rangle$, then there is nothing to be proved. Let $g \in \langle h \rangle$ or equivalently $g= h^\alpha$, for some $\alpha$. But then $o(g) \ | \ o(h)$, which implies that $o(g) = o(h)$ because $o(g)\in \mu(G)$, and so $\langle g \rangle= \langle h \rangle$. Therefor, $\langle g \rangle$ is a connected component in $\mathcal{P}^*(G)$, and so by Theorem \ref{MR1}, 
$\kappa(\langle g \rangle)=p^{p-2} \ | \ \kappa(G)$, as required. 
\end{proof}

The following result is  not limited to only connected algebraic graphs, and holds for all connected simple graphs. 

\begin{Lemma}\label{DEG} Let $\Gamma$ be a simple graph and $\{v_1, v_2, \ldots, v_k\}\subseteq V(\Gamma)$. If the vertices $v_1, v_2, \ldots, v_k$ have full-degree, then $|V(\Gamma)|^k \ | \ \det((\mathbf{J}+\mathbf{Q})$.
\end{Lemma}
\begin{proof} Let $|V(\Gamma)|=n$. Since $v_1, v_2, \ldots, v_k$ have full-degree, the adjacency matrix $\mathbf{A}(\Gamma)$ and the diagonal matrix $\mathbf{\Delta}(\Gamma)$ have  the following
 structures:

$$\mathbf{A}(\Gamma)
=\left(
\begin{array}{c | c} 
\mathbf{J}_{k\times k}-\mathbf{I}_{k\times k}
 & \mathbf{J}_{k\times (n-k)}\\
&\\
\hline
&\\
 \mathbf{J}_{(n-k)\times k}&\mathbf{A}(\Gamma\setminus \{v_1, \ldots v_k\})\\
 \end{array} 
\right),
\ \ \ \& \ \ \ \ 
\mathbf{\Delta}(\Gamma)
=\left(
\begin{array}{c | c} 
(n-1)\mathbf{I}_{k\times k}
 & \mathbf{O}_{k\times (n-k)}\\
&\\
\hline
&\\
 \mathbf{O}_{(n-k)\times k}&\Theta_{(n-k)\times (n-k)}\\
 \end{array} 
\right),
$$

\vspace{0.5cm}

\noindent where $\Theta_{(n-k)\times (n-k)}$ is the diagonal matrix whose diagonal entries are the degree of vertexes $V(\Gamma)\setminus \{v_1, \ldots v_k\}$ in $\Gamma$.
Therefore the  matrix $\mathbf{J}+\mathbf{Q}=\mathbf{J}+(\mathbf{\Delta}(\Gamma)-\mathbf{A}(\Gamma)) $ has the following structure:

\vspace{0.5cm}

$$\mathbf{J}+\mathbf{Q}
=\left(
\begin{array}{c | c} 
n\mathbf{I}_{k\times k}
 & \mathbf{O}_{k\times (n-k)}\\
&\\
\hline
&\\
 \mathbf{O}_{(n-k)\times k}&\mathbf{J}_{(n-k)\times(n-k)}-(\mathbf{A}(\Gamma\setminus \{v_1, \ldots v_k\})-\Theta_{(n-k)\times (n-k)})\\
 \end{array} 
\right).
$$

\vspace{0.5cm}

\noindent Now, by the structure  of  $\mathbf{J}+\mathbf{Q}$, we have  
$$\det(\mathbf{J}+\mathbf{Q})=\det(n\mathbf{I}_{k\times k})\cdot \det(\mathbf{J}_{(n-k)\times(n-k)}-(\mathbf{A}(\Gamma\setminus \{v_1, \ldots v_k\})-\Theta_{(n-k)\times (n-k)})),$$
and so $n^k\ | \ \det(\mathbf{J}+\mathbf{Q})$, as a required. 
\end{proof}

In Lemma \ref{B2}, we find a divisor for $\kappa(G)$, when a prime divisor of $G$ be in the $\mu(G)$. In the next theorem,  we extend the result for any number in the $\mu(G)$, and represent a divisor for $\det(\mathbf{J}+\mathbf{Q}(\mathcal{P}(G))$. 

\begin{Theorem}\label{mu3} Let $G$ be a finite group and $m\in \mu(G)$. Then $|G|\cdot m^{\phi(m)} \ | \ \det(\mathbf{J}+\mathbf{Q}(\mathcal{P}(G))$. 
\end{Theorem}
\begin{proof} Let $|G|=n$. By Lemma \ref{DEG}, since the vertex $1$ has full-degree, we have 
 $$\det(\mathbf{J}_{n\times n}+\mathbf{Q}(\mathcal{P}(G)))=n \cdot \det(\mathbf{J}_{(n-1)\times (n-1)}+(\mathbf{\Delta}-\mathbf{A}(\mathcal{P}^*(G)))), $$
where $\mathbf{\Delta}$ is the diagonal matrix whose diagonal entries are the degree of vertexes in  $\mathcal{P}^*(G)$ in $\mathcal{P}(G)$.

Suppose that $g\in G$ be an element of order $m$. By the definition of power graph, since 
$x\in G$ is adjacent to $g$ if and only if $\langle x\rangle\leq \langle g \rangle$, 
therefore $\mathbf{A}(\mathcal{P}^*(G))$ and
 $\mathbf{\Delta}$ have the following block-matrix structure:

\vspace{0.5cm}

$$
\mathbf{A}(\mathcal{P}^*(G))=\left(
\begin{array}{c | c c c}
(\mathbf{J}-\mathbf{I})_{\phi(m)\times \phi(m)} 
 & \mathbf{J}_{\phi(m)\times ((m-1)-\phi(m))}   \ \ \ \ \  \rvline&  \mathbf{O}_{\phi(m)\times( n-m)}  \\ 
&&& \\
\hline
&&& \\
   \mathbf{J}_{((m-1)-\phi(m))\times\phi(m)}  &  &  \\
&&& \\
\cline{1-1}
& \ \ \ \ \ \ \ \ \ \ \ \    \Omega_ {((n-1)-\phi(m))\times((n-1)-\phi(m))}&\\
  \mathbf{O}_{(n-m)\times \phi(m)}&&
 \end{array}
\right) 
,$$
\vspace{0.2cm}

\noindent where $ \Omega_ {((n-1)-\phi(m))\times((n-1)-\phi(m))}=\mathbf{A}(\mathcal{P}(G)\setminus \langle g\rangle)$, and

\vspace{0.5cm}

$$
\mathbf{\Delta}=\left( 
\begin{array}{c | c c c} 
(m-1)\mathbf{I}_{\phi(m)\times \phi(m)} 
 &  \mathbf{O}_{\phi(m)\times ((n-1)-\phi(m))}  \\ 
&&& \\
\hline
&&& \\
\mathbf{O}_{((n-1)-\phi(m))\times \phi(m)}& \ \ \ \ \ \ \ \ \ \ \ \    \Lambda_ {((n-1)-\phi(m))\times ((n-1)-\phi(m))}&\\
  \end{array}
\right) 
,$$

\vspace{0.5cm}

\noindent where $\Lambda_ {((n-1)-\phi(m))\times ((n-1)-\phi(m))}$ is the diagonal matrix whose diagonal entries are the degree of elements of $\mathcal{P}(G)\setminus \langle g\rangle$ in $\mathcal{P}(G)$. 

On the other hand, by considering  the Laplacian matrix $\mathbf{Q}=\mathbf{\Delta}-\mathbf{A}(\mathcal{P}^*(G)))$,
we must have the following block-matrix structure for $\mathbf{J}_{(n-1)\times (n-1)}+\mathbf{Q}$:

$$
\mathbf{J}_{(n-1)\times (n-1)}+\mathbf{Q}=\left(
\begin{array}{c | c c c} 
m \mathbf{I}_{\phi(m)\times \phi(m)} 
 & \mathbf{O}_{\phi(m)\times((m-1)-\phi(m))}   \ \ \ \ \  \rvline&  \mathbf{J}_{\phi(m)\times (n-m)}  \\ 
&&& \\
\hline
&&& \\
   \mathbf{O}_{((m-1)-\phi(m))\times\phi(m)}  &  &  \\
&&& \\
\cline{1-1}
& \ \ \ \ \ \ \ \ \ \ \ \    \Theta_ {((n-1)-\phi(m))\times ((n-1)-\phi(m))}&\\
  \mathbf{J}_{(n-m)\times \phi(m)}&&
 \end{array} 
\right) 
,$$
where 
$${\Theta}_{((n-1)-\phi(m))\times ((n-1)-\phi(m))}=\mathbf{J}_{{((n-1)-\phi(m))\times ((n-1)-\phi(m))}}+ \ \ \ \ \ \ \ \ \ \ \ \ \ \  \ \ \ \ \ \ \ \ \ \ \ \ \ \  \ \ \ \ \ \ \ \ \ \ \ \ \ \ $$
$$ \ \ \ \ \ \ \ \ \ \ \ \ \ \  \ \ \ \ \ \ \ \ \ \ \ \ \ \  \ \ \ \ \ \ \ \ \ \ \ \ \ \    \ \ \ \ \ \ \ \ \ \ \ \ \ \ (\Lambda_ {((n-1)-\phi(m))\times ((n-1)-\phi(m))}-\Omega_ {((n-1)-\phi(m))\times((n-1)-\phi(m))} ).$$

\vspace{0.3cm}

In the sequel, $Ri$ and $Cj$ respectively designate the row $i$ and the column $j$ of the matrix $\mathbf{J}_{(n-1)\times (n-1)}+\mathbf{Q}$.
We apply the following row and column operations in the $\det(\mathbf{J}_{(n-1)\times (n-1)}+\mathbf{Q})$.
We subtract row $R_1$ from row $R_i$, for $i=2, 3, \ldots, \phi(m)$ and subsequently we add
column $C_j$ to column $C_1$, for $j=2, 3, \ldots, \phi(m)$. It is not too difficult to see that, step by step, we have:

$$\det(\mathbf{J}_{(n-1)\times (n-1)}+\mathbf{Q})
=\det (\left(
\begin{array}{c | c c c} 
m \mathbf{I}_{\phi(m)\times \phi(m)} 
 & \mathbf{O}_{\phi(m)\times((m-1)-\phi(m)) }   \ \ \ \ \  \rvline&
\begin{array}{ccc}
 1 &\dots &1 \\
0&\dots &0  \\
\vdots &\vdots &\vdots \\
0 & \dots &0 
\end{array}  \\ 
&&& \\
\hline
&&& \\
   \mathbf{O}_{((m-1)-\phi(m))\times\phi(m)) }  &  &  \\
&&& \\
\cline{1-1}
& \ \ \ \ \ \ \ \ \ \ \ \    \Theta_ {((n-1)-\phi(m))\times ((n-1)-\phi(m))}&\\
 \begin{array}{cccc}
 \phi(m) &1&\dots &1 \\
\vdots& \vdots &\dots &\vdots   \\
\phi(m) &1 &\dots &1
\end{array}  \\ &&
 \end{array} 
\right)),
$$

\vspace{0.5cm}

\noindent again, we  subtract $\frac {1}{m} C_1$ from $C_j$, for $j=m,  \ldots, n-1$, and so 
\vspace{0.5cm}

$$\det(\mathbf{J}_{(n-1)\times (n-1)}+\mathbf{Q})
=\det (\left(
\begin{array}{c | c c c} 
m \mathbf{I}_{\phi(m)\times \phi(m)} 
 & \mathbf{O}_{((n-1)-\phi(m))\times((n-1)-\phi(m)) }   &\\
&&& \\
\hline
&&& \\
   \mathbf{O}_{((m-1)-\phi(m))\times\phi(m) }  &  &  \\
&&& \\
\cline{1-1}
& \ \ \ \ \ \ \ \ \ \ \ \    \Theta_ {((n-1)-\phi(m))\times ((n-1)-\phi(m))}-\Upsilon\\
 \begin{array}{cccc}
 \phi(m) &1&\dots &1 \\
\vdots& \vdots &\dots &\vdots   \\
\phi(m) &1 &\dots &1
\end{array}  \\ &&
 \end{array} 
\right)),
$$
where $\Upsilon$ has the following structure

$$
\Upsilon= \left(
\begin{array}{c|c } 
&1 \ \ \ \ \  \ldots \ \ \ \ \  1\\
& \\
\cline{2-2}
\mathbf{O}_{((n-1)-\phi(m))\times ((m-1)-\phi(m)) } &  \\
&\\
& \mathbf{O}_{((m-1)-\phi(m))\times(n-m) }    \\
& \\
\cline{2-2}
&\\
 &\frac{\phi(m)}{m}  \mathbf{J}_{(n-m)\times((n-m) } \\
 \end{array} 
\right),
$$

\vspace{0.5cm}

\noindent and so $$\det(\mathbf{J}_{(n-1)\times (n-1)}+\mathbf{Q})= m^{\phi(m)} \cdot \det( \Theta_ {((n-1)-\phi(m))\times ((n-1)-\phi(m))}-\Upsilon).$$

\vspace{0.3cm}

 On the other hand, by above mention discussion, 
$\det(\mathbf{J}_{n\times n}+\mathbf{Q}(\mathcal{P}(G)))=n \cdot \det(\mathbf{J}_{(n-1)\times (n-1)}+\mathbf{Q}$. 
Therefore $$n\cdot m^{\phi(m)} \  | \ \det(\mathbf{J}_{n\times n}+\mathbf{Q}(\mathcal{P}(G))),$$ as we claimed. 
\end{proof}

\begin{Corollary} \label{B7} Let $G$ be a group and $g\in G$. If the degree of $g$ is $k$ in $\mathcal{P}(G)$, then $$|G|\cdot (k+1)^{\phi(|g|)} \ | \ \det(\mathbf{J}_{n\times n}+\mathbf{Q}(\mathcal{P}(G))).$$
\end{Corollary}
\begin{proof} By the exact same way in the proof of Theorem \ref{B6}, the result is straightforward. 
\end{proof}

%%%%%%%%%%%%%%%%%%%%%%%%%%%%%%%%%%%%%%%%%%%%%%

At the end, as an application of our main results, we are going to prove Theorem \ref{A6}. As a matter of fact, we verify that  $h_{\mathcal {S}} (A_6\cong L_2(9)) = 1$,
in the class $\mathcal{S}$ of all
finite simple groups.

\vspace{0.5cm}

\begin{proof} For the proving of Theorem \ref{A6}, assume that  $G\in \mathcal{S}$, with  $\kappa(G)=\kappa(A_6) = 2^{180}\cdot3^{40}\cdot5^{108}$ (see Theorem \ref{mnn}).
First of all,  $G$ is a non-abelian simple group. Otherwise, $\kappa(G)=\kappa(\mathbb{Z}_p)=p^{p-2}$, for some prim number $p$, which is a contradiction. 

In the next, we claim that  
$\pi(G)\subseteq \{2, 3, 5, 7, 11\}$.
By Lemma \ref{MR1}, we have $c_p\geqslant p+1$, where $p\in \pi(G)$
and $c_p$ is the number of cyclic subgroups of order $p$ in $G$. Therefore, by Lemma \ref{MR}, 

$$13^{11\cdot14} \gneqq2^{180}\cdot3^{40}\cdot5^{108}=\kappa(G) > \kappa({\mathbb Z}_p)^{c_p}\geqslant \kappa ({\mathbb Z}_p)^{p+1}=p^{(p-2)(p+1)},$$
which leads us to the conclusion. 

Finally we are going to show that $G\cong A_6\cong L_2(9)$.
If $7$ or $11$ be an element in $\mu(G)$, then by Lemma \ref{B2}, $7^5$, or $11^9$ divide $\kappa(G)$, which is a contradiction. 
Now, by results collected in \cite{zav}, 
$G$ is isomorphic to one of the groups $A_5\cong L_2(4)\cong L_2(5)$, $A_6\cong L_2(9)$, $S_4(7)$. 
By Theorem \ref{AA5}, $G\ncong A_5$.
 If $G\cong S_4(7)$, then by Theorem \ref{mu3}, since $56\in \mu(G)$, 
$$|G|\cdot (56)^\phi(56)= |S_4(7)|\cdot 2^{72}\cdot 7^{24}\ | \ \det(\mathbf{J}+\mathbf{Q}), $$
which concludes that (by Theorem \ref{TE1}) $7^{20}\ | \ \kappa(G)$, again we have a contradiction. 
Therefore $G\cong A_6$, and the proof has been completed. 
\end {proof}

%%%%%%%%%%%%%%%%%%%%%%%%%%%%%%%%%%%%%%%%%%%%%%%%%

	%\cite{khalil2002nonlinear}
	%\lipsum[6-9]
	%\bibliographystyle{plain}
	%\bibliography{BibFile}
	
\end{document}